\documentclass[11pt]{article}

\usepackage{amsmath,amssymb,amsthm,microtype,geometry}

\newtheorem{theorem}{Theorem}[section]

\newtheorem{lemma}{Lemma}[section]

\newtheorem{proposition}{Proposition}[section]
\newtheorem{corollary}{Corollary}[section]

\title{A sharp $p$-subadditive bound for the $\ell_p$ Hausdorff distance from convex hull}
\author{Mark Meyer}

\begin{document}

\maketitle
\begin{abstract}
    We study the $\ell_p$ Hausdorff distance from convex hull, which for a compact set $A\subset \mathbb{R}^n$ is defined by
    \begin{equation*}
        d^{(\ell_p)}(A):=\sup_{x\in \textup{conv}(A)}\inf_{a\in A}\|x-a\|_p.
    \end{equation*}
    In the planar case $n=2$, we study the problem of finding the optimal constant $C_p$ such that 
    \begin{align*}
        d^{(\ell_p)}(A+B)^p\leq C_p\left(d^{(\ell_p)}(A)^p+d^{(\ell_p)}(B)^p\right)
    \end{align*}
    for all nonempty compact $A,B\subset\mathbb{R}^2$. We resolve this question, proving that
    \begin{align*}
        C_p=\max\{1,2^{p-2}\}.
    \end{align*}
\end{abstract}

\section{Introduction}

The Minkowski sum of sets $A$ and $B$ in $\mathbb{R}^n$ is defined by
\begin{align*}
    A+B:=\{a+b:a\in A,b\in B\}.
\end{align*}
For a parameter $1\leq p\leq \infty$, define the $\ell_p$ Hausdorff distance from convex hull by
\begin{align*}
    d^{(\ell_p)}(A):=\sup_{x\in \textup{conv}(A)}\inf_{a\in A}\|x-a\|_p.
\end{align*}
Intuitively, the quantity $d^{(\ell_p)}(A)$ is the farthest distance to the set $A$ from within its convex hull. In fact, the name Hausdorff distance comes from the fact that $d^{(\ell_p)}(A)$ is exactly the distance between $A$ and $\textup{conv}(A)$, measured by the Hausdorff metric induced by the norm $\|\cdot\|_p$.

We will study the planar case, and prove the following bound.

\begin{theorem}\label{theorem:p-subadditive_bound}
    Fix $1\leq p<\infty$. If $A,B\subset\mathbb{R}^2$ are nonempty compact sets, then
    \begin{align}\label{eq:p-subadditive_bound_formula}
        d^{(\ell_p)}(A+B)^p\leq \max\{1,2^{p-2}\}\left(d^{(\ell_p)}(A)^p+d^{(\ell_p)}(B)^p\right).
    \end{align}
    The constant $\max\{1,2^{p-2}\}$ is sharp.
\end{theorem}

The motivation for Theorem \ref{theorem:p-subadditive_bound} originates from the Dyn--Farkhi conjecture. In the paper \cite{dyn1}, Dyn and Farkhi conjectured that 
\begin{align*}
    d^{(\ell_2)}(A+B)^2\leq d^{(\ell_2)}(A)^2+d^{(\ell_2)}(B)^2
\end{align*}
holds for any nonempty compact $A,B\subset\mathbb{R}^n$. Considering a similar result proved by Cassels in \cite{cassels1}, it was natural to make this conjecture. However, the Dyn--Farkhi conjecture was disproved in the survey \cite{fradelizi1} of Fradelizi, Madiman, Marsiglietti, and Zvavitch for $n\geq 3$. Their counterexample relied fundamentally on the fact that $A\neq B$, so they further asked if the conjecture holds in the symmetric case $A=B$. Recent work of van Hintum \cite{vanHintum} showed the conjecture is false even in the symmetric case. The case $n=1$ is trivial, and as observed in \cite[Lemma 2.4]{meyer1} and \cite[Theorem 7.5]{fradelizi1}, much stronger bounds are possible.

This leaves $n=2$, which turns out to be the exceptional case. We proved in \cite{meyer2} that the Dyn--Farkhi conjecture is true, and we observed that the bound is sharp in a certain global sense. The purpose of this paper is to address the next natural question of finding the optimal $p$-dependent constant for which the function $(d^{(\ell_p)})^p$ is subadditive. Theorem \ref{theorem:p-subadditive_bound} answers exactly this question, showing that $(d^{(\ell_p)})^p$ is in fact subadditive when $1\leq p\leq 2$, and that when $p>2$ the optimal constant is $2^{p-2}$.

A particular aspect of Theorem \ref{theorem:p-subadditive_bound} worth further investigation is the sharpness of the optimal constant. The following proposition puts this into precise terms.

\begin{proposition}\label{proposition:sharpness_of_constant}
    The constant 
    \begin{align*}
        \max\{1,2^{p-2}\}
    \end{align*}
    in Theorem \ref{theorem:p-subadditive_bound} is sharp in the following sense:
    \begin{enumerate}
        \item If $1\leq p\leq 2$, then for any nonnegative $\alpha,\beta$, there exist compact sets $A,B\subset\mathbb{R}^2$ such that 
        \begin{align*}
            d^{(\ell_p)}(A)=\alpha,\qquad d^{(\ell_p)}(B)=\beta,
        \end{align*}
        and equality holds in \eqref{eq:p-subadditive_bound_formula}.\\
        \item If $p>2$, then for any $\alpha>0$, there exist compact sets $A,B\subset\mathbb{R}^2$ such that 
        \begin{align*}
            d^{(\ell_p)}(A)=d^{(\ell_p)}(B)=\alpha,
        \end{align*}
        and equality holds in \eqref{eq:p-subadditive_bound_formula}.
    \end{enumerate}
\end{proposition}

In summary, Proposition \ref{proposition:sharpness_of_constant} shows that in the case $1\leq p\leq 2$, not only is $(d^{(\ell_p)})^p$ subadditive without a $p$-dependent constant, but equality is possible for every pair $(d^{(\ell_p)}(A),d^{(\ell_p)}(B))=(\alpha,\beta)$, just as was observed in the paper \cite{meyer2} with the solution to the Dyn--Farkhi conjecture. However, when $p>2$, the known equality examples lie on the diagonal $(d^{(\ell_p)}(A),d^{(\ell_p)}(B))=(\alpha,\alpha)$, suggesting that away from this diagonal there may be an even sharper bound for which equality is possible for any $(\alpha,\beta)$. In fact, if we consider the case where one of the sets is convex, then it is easy to show that 
\begin{align*}
    d^{(\ell_p)}(A+B)^p\leq \max\{d^{(\ell_p)}(A)^p,d^{(\ell_p)}(B)^p\},
\end{align*}
which is strictly less than the upper bound \eqref{eq:p-subadditive_bound_formula} when $p>2$ and one of the sets $A$ or $B$ is nonconvex.

We also briefly mention the case $p=\infty$. Specifically when $n=2$, the unit $\ell_1$-ball and the unit $\ell_{\infty}$-ball differ by an invertible linear transformation, so the problem of finding the sharp upper bound for $d^{(\ell_{\infty})}$ is identical to the problem of finding the sharp upper bound for $d^{(\ell_1)}$.

\begin{corollary}\label{l_infty_bound}
    If $A,B\subset\mathbb{R}^2$ are nonempty compact sets, then
    \begin{align*}
        d^{(\ell_{\infty})}(A+B)\leq d^{(\ell_{\infty})}(A)+d^{(\ell_{\infty})}(B).
    \end{align*}
    Moreover, this bound is sharp in the same global sense as described in Proposition \ref{proposition:sharpness_of_constant} for the case $p=1$.
\end{corollary}

The Dyn--Farkhi conjecture is associated with a broader field of study, referred to as the convexification effect of Minkowski summation. The general idea is that the sum $A+\dots+A$ taken $m$-times divided by an appropriate scaling factor, usually $m$, tends to resemble $\textup{conv}(A)$ with respect to different measures as $m$ tends to infinity. This study began with the results of Shapley--Folkman--Starr \cite{starr1}, and Emerson--Greenleaf \cite{emerson1}, which give estimates for the Hausdorff metric. Since then, this phenomenon has been studied in a variety of different formats such as \cite{schneider1,cassels1,wegmann1,fradelizi1,FMMZ-2016,vHST-2020,A-2021,AK-2025,FLZ-2022}, to name just a few.

\subsection{Acknowledgments and funding} We thank Matthieu Fradelizi, Robert Fraser, Buma Fridman, and Olivier Guédon for helpful discussions. We are grateful to an anonymous referee for pointing out the shortened proof of Lemma \ref{lemma:triangle_lemma}, significantly improving the quality of the paper. This work was completed under the financial support of the National Science Foundation through the MSPRF program (award number: 2502794). 

\section{Proofs}

We aim to first prove Lemma \ref{lemma:parallelogram_reduction_K_norm}. We note that this lemma was essentially proved in \cite{meyer2}. However, we will include the full proof in this paper, since we give a much more complete argument that also handles the case where $A$ and $B$ are both of dimension one, which we did not show in the original paper.

We will prove this result for an arbitrary norm whose closed unit ball is an origin symmetric convex body. Recall that a convex body $K\subset\mathbb{R}^n$ is a compact convex set with nonempty interior. The Minkowski functional of $K$ is defined by
\begin{align*}
    \|x\|_K=\inf\{\lambda>0:x\in \lambda K\},\qquad x\in \mathbb{R}^n.
\end{align*}
A convex body $K$ is symmetric about the origin if $K=-K$, and it is well known that if $K$ is symmetric about the origin, then $\|\cdot \|_K$ is a norm. The generalized Hausdorff distance from convex hull is defined in \cite{fradelizi1} by
\begin{align*}
    d^{(K)}(A):=\sup_{x\in \textup{conv}(A)}d^{(K)}(x,A),\qquad d^{(K)}(x,A):=\inf_{a\in A}\|x-a\|_K.
\end{align*}

In the following lemma, we consider the situation where a hyperplane intersects the convex hull of a compact set, and look for the minimal distance between two points of the compact set that belong to different half-spaces.

\begin{lemma}\label{lemma:opposite_hyperplane}
    Let $K\subset\mathbb{R}^n$ be a convex body that is symmetric about the origin, let $B\subset\mathbb{R}^n$ be a nonempty compact set, and let $H$ be a hyperplane that intersects $\textup{conv}(B)$. If $H_1$ and $H_2$ are the two closed half-spaces that are bounded by $H$, then the minimum
    \begin{align*}
        \inf\{\|b_1-b_2\|_K:b_1\in B\cap H_1,b_2\in B\cap H_2\}
    \end{align*}
    is attained, and every minimizing pair satisfies 
    \begin{align*}
        \|b_1-b_2\|_K\leq 2d^{(K)}(B).
    \end{align*}
\end{lemma}

\begin{proof}
    There exists $c\in \mathbb{R}$ and $u\in \mathbb{R}^n$ such that $H=\{h:\langle h,u\rangle =c\}$. Since $H$ intersects $\textup{conv}(B)$, we have
    \begin{align*}
        \min_{b\in B}\langle b,u\rangle\leq c\leq \max_{b\in B}\langle b,u\rangle,
    \end{align*}
    and therefore $B\cap H_1$ and $B\cap H_2$ are nonempty. Since these sets are compact, the infimum is attained for some $b_1$ and $b_2$. Denote the above infimum by $M$, and set $m:=(b_1+b_2)/2$. If $M>0$, then the interior of the $K$-ball $m+(M/2)K$ does not contain any points of $B$. To see this, choose $b\in B$ that lies in the interior of the $K$-ball. Then $b$ lies in a half-space $H_i$ and its distance to the point $b_j$ not belonging to $H_i$ is
    \begin{align*}
        \|b-b_j\|_K\leq \|b-m\|_K+\|m-b_j\|_K<\frac{M}{2}+\frac{M}{2}=M,
    \end{align*}
    which is a contradiction. Therefore,
    \begin{align*}
        \frac{M}{2}=d^{(K)}(m,B)\leq d^{(K)}(B),
    \end{align*}
    and the claim follows. Note that the claim is immediate when $M=0$.
\end{proof}

The next lemma shows how to write the sum of convex sets as an ``interior" part and an ``exterior" part. 

\begin{lemma}\label{lemma:sum_decomposition}
    Let $L,V\subset\mathbb{R}^n$ be compact convex sets with $0\in L\cap V$ and $\textup{aff}(L)\subset \textup{aff}(V)$. Then
    \begin{align*}
        L+V=V\cup (\partial V+L),
    \end{align*}
    where $\partial V$ denotes the relative boundary of $V$.
\end{lemma}

\begin{proof}
    Since $0\in L$, it is clear that the right side is contained in the left side. Let $x=l+v\in L+V$. There is nothing to prove if $x\in V$. If $x\notin V$, then the interval $[v,x]$ intersects the boundary of $V$ at some point $z=v+tl$, $0\leq t<1$. Then
    \begin{align*}
        x=z+(1-t)l\in z+L\subset \partial V+L,
    \end{align*}
    where we used that $0,l\in L$ imply $(1-t)l\in L$, which is true because $L$ is convex.
\end{proof}

The following lemma describes a technical situation that arises in the proof of Lemma \ref{lemma:parallelogram_reduction_K_norm}.

\begin{lemma}\label{lemma:line_passing_technical_lemma}
    Let $A,B\subset\mathbb{R}^2$ be compact sets of positive dimension, and suppose that $x\in \textup{conv}(A+B)$ such that 
    \begin{align*}
        x\notin A+\textup{conv}(B),\qquad x\notin \textup{conv}(A)+B.
    \end{align*}
    Assume that $a_1,a_2$ are distinct elements of $A$, and that $x\in [a_1,a_2]+\textup{conv}(B)$. Set $u:=a_2-a_1$, and let $L:=x+\textup{span}(u)$. Then
    \begin{align*}
       C:= L\cap (\textup{conv}(B)+a_1)\neq \varnothing,\qquad C+[0,u]=L\cap (\textup{conv}(B)+[a_1,a_2]).
    \end{align*}
    Moreover, if $y\in C$, then $x\in y+[0,u]$, and $C\cap (B+\{a_1,a_2\})=\varnothing$.
\end{lemma}

\begin{proof}
    The set $\textup{conv}(B)+a_1$ is the union of all of its slices by lines parallel to $u$. The sum of $[0,u]$ with a set contained in one of these lines remains in the line. Therefore, the identity
    \begin{align*}
        \textup{conv}(B)+[a_1,a_2]=(\textup{conv}(B)+a_1)+[0,u]
    \end{align*}
    restricted to the line $L$ is exactly the desired identity. The set $C$ must be nonempty, because the right side of the identity contains $x$.

    For fixed $y\in C$, there exist $\alpha,\beta\geq 0$ such that
    \begin{align*}
        C=y+[-\alpha,\beta]u.
    \end{align*}
    Adding $u$ to $C$ gives
    \begin{align*}
        L\cap (\textup{conv}(B)+a_2)=y+[1-\alpha,1+\beta]u.
    \end{align*}
    Now, $x=y+\lambda u$ for some $\lambda\in [-\alpha,1+\beta]$. Since $x$ belongs to neither $\textup{conv}(B)+a_1$ nor $\textup{conv}(B)+a_2$, we get 
    \begin{align*}
        \beta< \lambda< 1-\alpha,
    \end{align*}
    so that $x\in y+[0,u]$.

    Finally, if $y=b+a_1\in C$, with $b\in B$, then $x\in b+[a_1,a_2]\subset B+\textup{conv}(A)$, which is a contradiction. If $y=b+a_2\in C$, then $y-u=b+a_1\in C$, and we arrive at the same contradiction. Therefore, $C$ does not contain any point of $B+\{a_1,a_2\}$.
\end{proof}

We have the necessary facts to prove the parallelogram reduction.

\begin{lemma}\label{lemma:parallelogram_reduction_K_norm}
        Let $K\subset\mathbb{R}^2$ be a convex body that is symmetric about the origin, and let $A,B\subset\mathbb{R}^2$ be nonempty compact sets. If $x\in \textup{conv}(A+B)$ such that 
        \begin{align*}
            x\notin \textup{conv}(A)+B,\qquad x\notin A+\textup{conv}(B),
        \end{align*}
        then there exist distinct $a_1,a_2\in A$ and distinct $b_1,b_2\in B$ with
        \begin{align*}
            \|a_1-a_2\|_K\leq 2d^{(K)}(A),\qquad \|b_1-b_2\|_K\leq 2d^{(K)}(B),
        \end{align*}
        such that 
        \begin{align*}
            x\in [a_1,a_2]+[b_1,b_2].
        \end{align*}
        Moreover, the intervals $[a_1,a_2]$ and $[b_1,b_2]$ are not parallel.
\end{lemma}

\begin{proof}
    First, suppose that $\textup{dim}(A)=\textup{dim}(B)=1$. Write $x=a+b$, where $a\in \textup{conv}(A)$ and $b\in \textup{conv}(B)$. The assumption on $x$ implies that $a\notin A$ and $b\notin B$. Let $a_1,a_2$, and $b_1,b_2$ be the endpoints of the gaps of $A$ and $B$ that contain $a$ and $b$, in the sense that
    \begin{align*}
        a\in [a_1,a_2],\qquad b\in [b_1,b_2],\qquad \|a_1-a_2\|_K\leq 2d^{(K)}(A),\qquad \|b_1-b_2\|_K\leq 2d^{(K)}(B).
    \end{align*}
    We used that the midpoint of the gap $[a_1,a_2]$ has $K$-distance to $A$ exactly $\|a_2-a_1\|_K/2$. By definition, this distance cannot exceed $d^{(K)}(A)$, which implies the bound $\|a_2-a_1\|_K\leq 2d^{(K)}(A)$. Use an identical argument to get the estimate on $\|b_2-b_1\|_K$. Moving on, we have $x\in [a_1,a_2]+[b_1,b_2]$. Moreover, the intervals $[a_1,a_2]$ and $[b_1,b_2]$ are not parallel. Otherwise, interchange the roles of $A$ and $B$ if necessary, and label the endpoints so that $u:=a_2-a_1$ and $\lambda u=b_2-b_1$, with $\lambda\geq 1$. Then
    \begin{align*}
        [a_1,a_2]+[b_1,b_2]=a_1+b_1+([0,\lambda u]\cup [\lambda u,(1+\lambda)u]),
    \end{align*}
    so that either 
    \begin{align*}
        x\in a_1+[b_1,b_2],\qquad \textup{or}\qquad x\in b_2+[a_1,a_2],
    \end{align*}
    which in either case is a contradiction since $a_1+[b_1,b_2]\subset A+\textup{conv}(B)$ and $b_2+[a_1,a_2]\subset B+\textup{conv}(A)$.

    Next, suppose that at least one of the sets has dimension $2$. Interchanging and translating the sets if needed, we assume that $\textup{dim}(A)=2$ and that $0\in A\cap B$. Lemma \ref{lemma:sum_decomposition} gives
    \begin{align*}
        \textup{conv}(A)+\textup{conv}(B)=\textup{conv}(A)\cup (\partial \textup{conv}(A)+\textup{conv}(B)).
    \end{align*}
    Using the above representation and the assumption on $x$, write $x=a+b$, where 
    \begin{align*}
        a\in \partial\textup{conv}(A)\backslash A,\qquad b\in \textup{conv}(B).
    \end{align*}
    If $H$ is a supporting line of $\textup{conv}(A)$ that passes through $a$, then $a\in \textup{conv}(A\cap H)$. Since $A\cap H$ is a compact subset of a line and $a\notin A$, there exist distinct $\tilde{a}_1,\tilde{a}_2\in A\cap H$ with $a\in [\tilde{a}_1,\tilde{a}_2]$, and so $x\in [\tilde{a}_1,\tilde{a}_2]+\textup{conv}(B)$.

    Let $L$ be the line that passes through $x$ and is parallel to the interval $[\tilde{a}_1,\tilde{a}_2]$. By Lemmas \ref{lemma:opposite_hyperplane} and \ref{lemma:line_passing_technical_lemma}, the set $C:=L\cap (\textup{conv}(B)+\tilde{a}_1)$ is nonempty, and there exist $b_1,b_2\in B$ such that $b_1+\tilde{a}_1$ and $b_2+\tilde{a}_1$ lie on opposite sides of $L$ and satisfy
    \begin{align*}
        \|b_1-b_2\|_K\leq 2d^{(K)}(B).
    \end{align*}
    Now, the interval $\tilde{a}_1+[b_1,b_2]$ intersects $L$ at a point $y$. Then $y\in C$, and by Lemma \ref{lemma:line_passing_technical_lemma}, we have $x\in y+[0,\tilde{a}_2-\tilde{a}_1]$, and so
    \begin{align}\label{eq:halfway_done}
        x\in [\tilde{a}_1,\tilde{a}_2]+[b_1,b_2].
    \end{align}
    Repeat the same argument with the roles of $A$ and $B$ reversed, starting with \eqref{eq:halfway_done} to find $a_1,a_2\in A$ such that
    \begin{align*}
        \|a_1-a_2\|_K\leq 2d^{(K)}(A),\qquad x\in [a_1,a_2]+[b_1,b_2].
    \end{align*}
    We use the same argument from the beginning of this proof to show that the intervals must be nonparallel.
\end{proof}

After reducing to parallelograms, we dissect the parallelogram into two triangles and use the following bound.

\begin{lemma}\label{lemma:triangle_lemma}
    Let $T\subset\mathbb{R}^2$ be a triangle such that two of the sides have $\ell_p$-length $a$ and $b$ respectively. Then
    \begin{align*}
        d^{(\ell_p)}(\textup{vert}(T))^p\leq \max\{1,2^{p-2}\}\left(\left(\frac{a}{2}\right)^p+\left(\frac{b}{2}\right)^p\right).
    \end{align*}
\end{lemma}

\begin{proof}
    Translate $T$ so that the sides of length $a$ and $b$ have a common vertex at $0$. Denote the other two vertices by $u$ and $v$ so that
    \begin{align*}
        \|u\|_p=a,\qquad \|v\|_p=b.
    \end{align*}
    Set
    \begin{align*}
        U:=\frac{u+v}{2},\qquad V:=\frac{u-v}{2}.
    \end{align*}
    For every $x\in T$, there exist nonnegative $s,t$, with $s+t\leq 1$ such that
    \begin{align*}
        x=su+tv=(s+t)U+(s-t)V.
    \end{align*}
    Define $\alpha:=s+t$ and $\beta:=s-t$, so that $0\leq \alpha\leq 1$, and $|\beta|\leq \alpha$. If $\beta\geq 0$, then compare the vertices $0$ and $u=U+V$, and use the inequality $\min\{X,Y\}\leq (X+Y)/2$, which holds for nonnegative $X,Y$, to get
    \begin{align*}
        d^{(\ell_p)}(x,\textup{vert}(T))&\leq \min\{\|x-0\|_p,\|x-u\|_p\}\\
        &\leq \min\{\alpha \|U\|_p+\beta\|V\|_p,(1-\alpha)\|U\|_p+(1-\beta)\|V\|_p\}\\
        &\leq \frac{\|U\|_p+\|V\|_p}{2}.
    \end{align*}
    If $\beta\leq 0$, then compare $0$ and $v=U-V$ to get the same bound. By convexity of $t\to t^p$, we get
    \begin{align*}
        d^{(\ell_p)}(\textup{vert}(T))^p\leq \frac{\|U\|_p^p+\|V\|_p^p}{2}.
    \end{align*}
    For $p=1$, the required estimate follows from the triangle inequality. Otherwise, when $1<p<\infty$, use Clarkson's inequality \cite[Theorem 2]{clarkson1} to get 
    \begin{align*}
        \|U\|_p^p+\|V\|_p^p&\leq \max\{2^{1-p},2^{-1}\}(a^p+b^p)\\
        &=\max\{2,2^{p-1}\}\left(\left(\frac{a}{2}\right)^p+\left(\frac{b}{2}\right)^p\right).
    \end{align*}
    The desired bound follows.
\end{proof}

\begin{proof}[Proof of Theorem \ref{theorem:p-subadditive_bound}]
    Let $x\in \textup{conv}(A+B)$. If either $x\in A+\textup{conv}(B)$ or $x\in \textup{conv}(A)+B$, then 
    \begin{align}\label{eq:max_bound}
        d^{(\ell_p)}(x,A+B)\leq\max\{d^{(\ell_p)}(A),d^{(\ell_p)}(B)\}.
    \end{align}
    Otherwise, $x$ satisfies the condition of Lemma \ref{lemma:parallelogram_reduction_K_norm}, and so $x$ is contained in a parallelogram whose vertex set is contained in $A+B$, and whose sides have $\ell_p$-length of at most $2d^{(\ell_p)}(A)$ and $2d^{(\ell_p)}(B)$. Dissect the parallelogram into two triangles, each of which has vertex set in $A+B$ and two of the sides have $\ell_p$-length at most $2d^{(\ell_p)}(A)$ and $2d^{(\ell_p)}(B)$. Let $T$ denote the triangle that contains $x$. Then by Lemma \ref{lemma:triangle_lemma}, we have
    \begin{align}\label{eq:p-subadditive-estimate}
        d^{(\ell_p)}(x,A+B)^p\leq d^{(\ell_p)}(\textup{vert}(T))^p\leq \max\{1,2^{p-2}\}\left(d^{(\ell_p)}(A)^p+d^{(\ell_p)}(B)^p\right).
    \end{align}
    The desired bound follows by taking the supremum over all $x\in \textup{conv}(A+B)$ and using the estimates \eqref{eq:max_bound} and \eqref{eq:p-subadditive-estimate}.
\end{proof}

\begin{proof}[Proof of Proposition \ref{proposition:sharpness_of_constant}]
    First, consider $1\leq p\leq 2$. If $\alpha=\beta=0$, take $A=B=\{0\}$. If $\alpha>0$ and $\beta=0$, set $A=\{(0,0),(2\alpha,0)\}$ and $B=\{0\}$, and similarly for the case $\alpha=0$ and $\beta>0$. If both $\alpha,\beta$ are positive, then set
    \begin{align*}
        A:=\{(0,0),(2\alpha,0)\},\qquad B:=\{(0,0),(0,2\beta)\}.
    \end{align*}
    Then $d^{(\ell_p)}(A)=\alpha$, $d^{(\ell_p)}(B)=\beta$. The sum $A+B$ is the vertex set of a rectangle with sides of $\ell_p$-lengths $2\alpha$ and $2\beta$ respectively. The farthest distance to a vertex is achieved at the center of the rectangle and it follows that \eqref{eq:p-subadditive_bound_formula} is equality.

    Next, consider $p>2$. Choose positive $\alpha$, and for $t:=2^{-1/p}\alpha$, set
    \begin{align*}
        A:=\{(0,0),(2t,2t)\},\qquad B:=\{(0,0),(-2t,2t)\}.
    \end{align*}
    Then
    \begin{align*}
        d^{(\ell_p)}(A)=d^{(\ell_p)}(B)=\alpha,
    \end{align*}
    so the right side of \eqref{eq:p-subadditive_bound_formula} is exactly $2^{p-1}\alpha^p$. The sum is
    \begin{align*}
        A+B:=\{(0,0),(2t,2t),(-2t,2t),(0,4t)\}.
    \end{align*}
    Set $m:=(0,2t)$, which is clearly in $\textup{conv}(A+B)$, being the midpoint of $(2t,2t)$ and $(-2t,2t)$. We compute
    \begin{align*}
        \|m-0\|_p^p=\|m-(2t,2t)\|_p^p=\|m-(-2t,2t)\|_p^p=\|m-(0,4t)\|_p^p=2^{p-1}\alpha^p,
    \end{align*}
    and it follows that
    \begin{align*}
        d^{(\ell_p)}(m,A+B)^p=2^{p-1}\alpha^p,
    \end{align*}
    so \eqref{eq:p-subadditive_bound_formula} is equality.
\end{proof}

\begin{proof}[Proof of Corollary \ref{l_infty_bound}]
    Define the invertible linear transformation $T:\mathbb{R}^2\rightarrow\mathbb{R}^2$ by
    \begin{align*}
        T(x,y)=(x+y,x-y).
    \end{align*}
    Then
    \begin{align*}
        \|T(x,y)\|_{\infty}=\|(x,y)\|_1,\qquad \textup{and}\qquad T(\textup{conv}(S))=\textup{conv}(T(S)),
    \end{align*}  
    for any nonempty compact $S\subset\mathbb{R}^2$, and therefore the formula
    \begin{align*}
        d^{(\ell_{\infty})}(T(S))=d^{(\ell_1)}(S)
    \end{align*}
    holds for any such $S$. By Theorem \ref{theorem:p-subadditive_bound} with $p=1$, and using the fact that $T(A+B)=T(A)+T(B)$, we have
    \begin{align*}
        d^{(\ell_\infty)}(T(A)+T(B))&=d^{(\ell_1)}(A+B)\\
        &\leq d^{(\ell_1)}(A)+d^{(\ell_1)}(B)\\
        &=d^{(\ell_{\infty})}(T(A))+d^{(\ell_\infty)}(T(B)).
    \end{align*}
    Because $T$ is invertible, the inequality holds for all nonempty compact sets. The bound is sharp as well, which we see by applying $T$ to the examples given in Proposition \ref{proposition:sharpness_of_constant} in the case $p=1$.
    
\end{proof}

\newpage
\bibliography{hausdorff_distance}

@article {A-2021,
    AUTHOR = {Artstein, Zvi},
     TITLE = {Convexification estimates for {M}inkowski averages in infinite
              dimensions},
   JOURNAL = {Pure Appl. Funct. Anal.},
  FJOURNAL = {Pure and Applied Functional Analysis},
    VOLUME = {6},
      YEAR = {2021},
    NUMBER = {4},
     PAGES = {709--718},
      ISSN = {2189-3756,2189-3764},
   MRCLASS = {52A05 (41A46 46A55 46B50 52A39)},
  MRNUMBER = {4302897},
}

@article {AK-2025,
    AUTHOR = {Artstein, Zvi and Kadets, Vladimir},
     TITLE = {B-convexity, convexification of {M}inkowski averages in a
              {B}anach space, and {SLLN} for random sets},
   JOURNAL = {J. Convex Anal.},
  FJOURNAL = {Journal of Convex Analysis},
    VOLUME = {32},
      YEAR = {2025},
    NUMBER = {1},
     PAGES = {61--70},
      ISSN = {0944-6532,2363-6394},
   MRCLASS = {46B20 (28B20)},
  MRNUMBER = {4950448},
}

@article {cassels1,
    AUTHOR = {Cassels, J. W. S.},
     TITLE = {Measures of the non-convexity of sets and the
              {S}hapley-{F}olkman-{S}tarr theorem},
   JOURNAL = {Math. Proc. Cambridge Philos. Soc.},
  FJOURNAL = {Mathematical Proceedings of the Cambridge Philosophical
              Society},
    VOLUME = {78},
      YEAR = {1975},
    NUMBER = {3},
     PAGES = {433--436},
      ISSN = {0305-0041,1469-8064},
   MRCLASS = {52A05 (90A99)},
  MRNUMBER = {385711},
MRREVIEWER = {Jacques\ Dubois},
       DOI = {10.1017/S0305004100051884},
       URL = {https://doi.org/10.1017/S0305004100051884},
}

@article {clarkson1,
    AUTHOR = {Clarkson, James A.},
     TITLE = {Uniformly convex spaces},
   JOURNAL = {Trans. Amer. Math. Soc.},
  FJOURNAL = {Transactions of the American Mathematical Society},
    VOLUME = {40},
      YEAR = {1936},
    NUMBER = {3},
     PAGES = {396--414},
      ISSN = {0002-9947,1088-6850},
   MRCLASS = {46B20},
  MRNUMBER = {1501880},
       DOI = {10.2307/1989630},
       URL = {https://doi-org.univ-eiffel.idm.oclc.org/10.2307/1989630},
}

@article {dyn1,
    AUTHOR = {Dyn, Nira and Farkhi, Elza},
     TITLE = {Set-valued approximations with {M}inkowski
              averages---convergence and convexification rates},
   JOURNAL = {Numer. Funct. Anal. Optim.},
  FJOURNAL = {Numerical Functional Analysis and Optimization. An
              International Journal},
    VOLUME = {25},
      YEAR = {2004},
    NUMBER = {3-4},
     PAGES = {363--377},
      ISSN = {0163-0563,1532-2467},
   MRCLASS = {41A30},
  MRNUMBER = {2072074},
       DOI = {10.1081/NFA-120039682},
       URL = {https://doi-org.univ-eiffel.idm.oclc.org/10.1081/NFA-120039682},
}

@article {emerson1,
    AUTHOR = {Emerson, W. R. and Greenleaf, F. P.},
     TITLE = {Asymptotic behavior of products {$C\sp{p}=C+\cdots +C$} in
              locally compact abelian groups},
   JOURNAL = {Trans. Amer. Math. Soc.},
  FJOURNAL = {Transactions of the American Mathematical Society},
    VOLUME = {145},
      YEAR = {1969},
     PAGES = {171--204},
      ISSN = {0002-9947,1088-6850},
   MRCLASS = {22.20},
  MRNUMBER = {249535},
MRREVIEWER = {L.\ Corwin},
       DOI = {10.2307/1995065},
       URL = {https://doi.org/10.2307/1995065},
}

@article {fradelizi1,
    AUTHOR = {Fradelizi, Matthieu and Madiman, Mokshay and Marsiglietti,
              Arnaud and Zvavitch, Artem},
     TITLE = {The convexification effect of {M}inkowski summation},
   JOURNAL = {EMS Surv. Math. Sci.},
  FJOURNAL = {EMS Surveys in Mathematical Sciences},
    VOLUME = {5},
      YEAR = {2018},
    NUMBER = {1-2},
     PAGES = {1--64},
      ISSN = {2308-2151,2308-216X},
   MRCLASS = {60E15 (11B13 60F15 94A17)},
  MRNUMBER = {3880220},
       DOI = {10.4171/EMSS/26},
       URL = {https://doi-org.univ-eiffel.idm.oclc.org/10.4171/EMSS/26},
}

@article {FLZ-2022,
    AUTHOR = {Fradelizi, Matthieu and L\'angi, Zsolt and Zvavitch, Artem},
     TITLE = {Volume of the {M}inkowski sums of star-shaped sets},
   JOURNAL = {Proc. Amer. Math. Soc. Ser. B},
  FJOURNAL = {Proceedings of the American Mathematical Society. Series B},
    VOLUME = {9},
      YEAR = {2022},
     PAGES = {358--372},
      ISSN = {2330-1511},
   MRCLASS = {52A40},
  MRNUMBER = {4474695},
MRREVIEWER = {Arnaud\ Marsiglietti},
       DOI = {10.1090/bproc/97},
       URL = {https://doi.org/10.1090/bproc/97},
}

@article {FMMZ-2016,
    AUTHOR = {Fradelizi, Matthieu and Madiman, Mokshay and Marsiglietti,
              Arnaud and Zvavitch, Artem},
     TITLE = {Do {M}inkowski averages get progressively more convex?},
   JOURNAL = {C. R. Math. Acad. Sci. Paris},
  FJOURNAL = {Comptes Rendus Math\'ematique. Acad\'emie des Sciences. Paris},
    VOLUME = {354},
      YEAR = {2016},
    NUMBER = {2},
     PAGES = {185--189},
      ISSN = {1631-073X,1778-3569},
   MRCLASS = {52A39 (52A20)},
  MRNUMBER = {3456896},
MRREVIEWER = {Julio\ Bernu\'es},
       DOI = {10.1016/j.crma.2015.12.005},
       URL = {https://doi.org/10.1016/j.crma.2015.12.005},
}

@article {meyer1,
    AUTHOR = {Meyer, Mark},
     TITLE = {Measuring the convexity of compact sumsets with the
              {S}chneider non-convexity index},
   JOURNAL = {J. Inequal. Appl.},
  FJOURNAL = {Journal of Inequalities and Applications},
      YEAR = {2025},
     PAGES = {Paper No. 103, 32},
      ISSN = {1029-242X},
   MRCLASS = {52A40},
  MRNUMBER = {4951555},
       DOI = {10.1186/s13660-025-03356-w},
       URL = {https://doi-org.univ-eiffel.idm.oclc.org/10.1186/s13660-025-03356-w},
}

@article {meyer2,
    AUTHOR = {Meyer, Mark},
     TITLE = {The {D}yn-{F}arkhi conjecture and the convex hull of a sumset
              in two dimensions},
   JOURNAL = {Proc. Amer. Math. Soc.},
  FJOURNAL = {Proceedings of the American Mathematical Society},
    VOLUME = {153},
      YEAR = {2025},
    NUMBER = {11},
     PAGES = {4889--4901},
      ISSN = {0002-9939,1088-6826},
   MRCLASS = {52A40 (52A10)},
  MRNUMBER = {4971580},
       DOI = {10.1090/proc/17313},
       URL = {https://doi-org.univ-eiffel.idm.oclc.org/10.1090/proc/17313},
}

@article {schneider1,
    AUTHOR = {Schneider, Rolf},
     TITLE = {A measure of convexity for compact sets},
   JOURNAL = {Pacific J. Math.},
  FJOURNAL = {Pacific Journal of Mathematics},
    VOLUME = {58},
      YEAR = {1975},
    NUMBER = {2},
     PAGES = {617--625},
      ISSN = {0030-8730,1945-5844},
   MRCLASS = {52A20},
  MRNUMBER = {380628},
MRREVIEWER = {H.\ W.\ Guggenheimer},
       URL = {http://projecteuclid.org/euclid.pjm/1102905693},
}

@article{starr1,
  title={Quasi-equilibria in markets with non-convex preferences},
  author={Starr, Ross M},
  journal={Econometrica: journal of the Econometric Society},
  pages={25--38},
  year={1969},
  publisher={JSTOR}
}

@article{vanHintum,
  title={The sharp threshold for {H}ausdorff convexification under {M}inkowski addition},
  author={van Hintum, Peter},
  journal={arXiv preprint arXiv:2606.10815},
  year={2026}
}

@article {vHST-2020,
    AUTHOR = {van Hintum, Peter and Spink, Hunter and Tiba, Marius},
     TITLE = {Sharp {$L^1$} inequalities for sup-convolution},
   JOURNAL = {Discrete Anal.},
  FJOURNAL = {Discrete Analysis},
      YEAR = {2023},
     PAGES = {Paper No. 7, 16},
      ISSN = {2397-3129},
   MRCLASS = {52A20 (26D15 49J53 90C25)},
  MRNUMBER = {4620306},
MRREVIEWER = {Jos\'e\ Vidal-N\'u\~nez},
       DOI = {10.19086/da},
       URL = {https://doi.org/10.19086/da},
}

@article {wegmann1,
    AUTHOR = {Wegmann, Rudolf},
     TITLE = {Einige {M}a\ss zahlen f\"ur nichtkonvexe {M}engen},
   JOURNAL = {Arch. Math. (Basel)},
  FJOURNAL = {Archiv der Mathematik},
    VOLUME = {34},
      YEAR = {1980},
    NUMBER = {1},
     PAGES = {69--74},
      ISSN = {0003-889X,1420-8938},
   MRCLASS = {52A05},
  MRNUMBER = {575552},
MRREVIEWER = {Diethard\ Pallaschke},
       DOI = {10.1007/BF01224931},
       URL = {https://doi.org/10.1007/BF01224931},
}
\bibliographystyle{abbrv}

\bigskip
\noindent Mark Meyer
\\
LAMA, Univ Gustave Eiffel, Univ Paris Est Creteil, 77447 Marne-la-Vall\'ee, France.
\\
E-mail address: mark.meyer@univ-eiffel.fr
\vspace{2mm}
\\

\end{document}